\documentclass[a4paper,11pt]{amsart}
\usepackage{amsthm,amsmath,amssymb,amscd}
\usepackage[latin1]{inputenc}
\usepackage{amssymb}
\usepackage{graphics}
\usepackage[all]{xy}

\newcommand{\K}{K\"{a}hler}

\newcommand{\CC}{{\mathbb{C}}}
\newcommand{\QQ}{{\mathbb{Q}}}

\newcommand{\PP}{{\mathbb{C} \mathbb{P}}}

\newcommand{\tr}{\operatorname{tr}}

\theoremstyle{plain}
\newtheorem{thm}{Theorem}[section]
\newtheorem{prop}[thm]{Proposition}

\newtheorem{cor}[thm]{Corollary}
\theoremstyle{definition}
\newtheorem{defn}[thm]{Definition}
\theoremstyle{remark}

\newtheorem{rem}[thm]{Remark}

\DeclareMathOperator{\rk}{rank}
\DeclareMathOperator{\codim}{Codim}
\DeclareMathOperator{\sing}{Sing}
\DeclareMathOperator{\supp}{Supp}
\DeclareMathOperator{\vol}{vol}
\DeclareMathOperator{\proj}{Proj}

\newcommand{\CM}{{\rm CM}}

\begin{document}

\title[Singularities and $K$-semistability]{Singularities and $K$-semistability}

\author{Claudio Arezzo}
\address{Abdus Salam International Center for Theoretical Physics \\
                  Strada Costiera 11 \\ 
         Trieste (Italy) and Dipartimento di Matematica \\
         Universit\`a di Parma \\ 
         Parco Area delle Scienze~53/A  \\ 
         Parma (Italy)}
\email{arezzo@ictp.it} 

\author{Alberto Della Vedova}
\address{Fine Hall, Princeton University, Princeton, NJ 08544 and
Dipartimento di Matematica \\
         Universit\`a di Parma \\ 
         Parco Area delle Scienze 53/A \\ 
         Parma (Italy)}
\email{della@math.princeton.edu}

\author{Gabriele La Nave}
\address{Department of Mathematics, University of Illinois, Urbana Il and \\ 
         Yeshiva University \\ 
         500 West 185 Street \\ 
         New York, NY}
\email{lanave@yu.edu}

\date{\today}

\begin{abstract}
In this paper we extend the notion of Futaki invariant to big and nef classes
so as to define a continuous function on the \K\ cone up to the 
boundary. We apply this concept to prove that reduced normal crossing singularities are sufficient to check $K$-semistability. 
A similar improvement on Donaldson's lower bound for Calabi energy
is given.
\end{abstract}

\maketitle

\section{Introduction}

One of the most fascinating problems in complex differential geometry is certainly the existence problem for canonical \K\ metrics in a fixed cohomology class (Einstein metrics are one important example). While, at least in the first place, one is primarily interested in studying such a problem on a smooth manifold, singular spaces almost immediately enter the scene at least for two important reasons.

On the one hand, when trying to construct such metrics by solving suitable Partial Differential Equations (such as, Monge-Amp\`ere, constant scalar curvature equation, Ricci or Calabi flows), with varying specific difficulties, one faces the questions of whether and how the solutions develop singularities. In this situations, after taking suitable geometric limits (such as Cheeger-Gromov or Gromov-Hausdorff), one is often forced to consider singular spaces.

The second reason is more subtle and is related to the Tian-Yau-Donaldson Conjecture (\cite{T95}, \cite{T97}, \cite{tianex}, \cite{Donaldson2002}) which predicts  that  the existence of special metrics is equivalent to a suitably adapted GIT stability notion of the corresponding algebraic {\em{polarized}} manifold. 

What is now believed to be the right stability notion entering this picture is the so called {\em{K-stability}} introduced by Tian (\cite{T95}, \cite{T97}) (see also Ding-Tian \cite{dt}) and later by Donaldson (\cite{Donaldson2002}), building on previous work by Futaki and Calabi. The mutual relationship between these, and other notions has been deeply investigated by Paul-Tian in \cite{pt} and some of their results will be recalled and used in our work.

Let us now recall
Donaldson's definition of $K$-stability.

\begin{defn}
\begin{enumerate}
\item
\label{defn::DF_invariant}
\noindent Let $(V,L)$ be a $n$-dimensional polarized variety or scheme. Given a one parameter subgroup $\rho: \mathbb C^* \to {\rm Aut}(V)$ with a linearization on $L$ and denoted by $w(V,L)$ the weight of the $\mathbb C^*$-action induced on $\bigwedge^{\rm top} H^0(V,L)$, we have the following asymptotic expansions as $k \gg 0$:
\begin{eqnarray}
\label{eq::asym_w(V,L^k)}h^0(V,L^k) &=& a_0 k^n + a_1 k^{n-1} + O(k^{n-2})\\
\label{eq::asym_h^0(V,L^k)} w(V,L^k) &=& b_0 k^{n+1} + b_1 k^{n} + O(k^{n-1}) 
\end{eqnarray}
The (normalized) \emph {Futaki invariant} of the action is $$ F(V,L, \rho) = \frac{b_1}{a_0}-\frac{b_0 \,a_1}{a_0^2}. $$
\item
\label{defn::TC}
A \emph{test configuration} $(X, L)\to \mathbb C$ of a polarized manifold $(M,A)$ consists of a scheme $X$ endowed with a $\mathbb C^*$-action that linearizes on a line bundle $L$ over $X$, and a flat $\mathbb C^*$-equivariant map $f: X \to \mathbb C$ (where $\mathbb C$ has the usual weight one $\mathbb C^*$-action) such that $L_0=L|_{f^{-1}(0)}$ is ample on $X_0=f^{-1}(0)$ and we have $(f^{-1}(1) , L|_{f^{-1}(1)}) \simeq (M,A^r)$ for some $r>0$. When $(M,A)$ has a $\mathbb C^*$-action $\rho: \mathbb C^* \to {\rm Aut}(M)$, a test configuration where $X = M \times \mathbb C$ and $\mathbb C^*$ acts on $X$ diagonally through $\rho$ is called \emph{product configuration}.
\item
\label{defn::semiKst}
\noindent The polarized manifold $(M,A)$ is \emph{$K$-semistable} if for each test configuration the Futaki invariant of the induced action on the central fiber $(X_0,L_0)$ is less than or equal to zero. $(M,A)$ is \emph{$K$-polystable} if moreover equality holds only in the case of a product configuration.
\end{enumerate}
\end{defn}

Even though $(M,A)$ is indeed a smooth manifold, a test configuration and its central fiber will in general be just a schemes. 
%

In order to make the total space of a test configuration (and its central fiber) less singular than a given one, it is natural to apply Mumford's semi-stable reduction Theorem (cfr. \ref{thm::ESSR}). Since the reduction
equips the central fiber with a line bundle which is merely big and nef --instead of a genuine polarization--one is forced
to extended the notion of Futaki invariant to schemes equipped with big and nef line bundles. 
It is also very natural--and useful when using degeneration arguments-- to try and do this so that the new invariant be continuous with respect to degenerations of the polarization within the K\"ahler cone. This is accomplished in Definition \ref{def::DF} and Proposition \ref{prop::DF_continuity}. It is interesting to stress that this definition shows the topological nature of this invariant.

The key point of our study is then to prove that two things may happen to the Futaki invariant
when applying Mumford's Theorem: either it jumps (and this has to happen if the initial singularities were very bad, see Corollary \ref{cor::control_Futaki})--in which case the new Futaki invariant is bigger than the starting one--
or it doesn't change (up to multiplication by the degree of the base change). This step is based on a crucial result by Ross-Thomas \cite{rt} where a similar jumping phenomenon is described for the Futaki invariant as in Donaldson's definition. When deforming the big and nef bundle on the central fiber to a genuine polarization, the above mentioned continuity immediately gives the following:

\begin{thm}
\label{genmain}
Given a test configuration $(X,L) \to \mathbb C$ for a smooth polarized manifold, then we have the following alternative:
\begin{enumerate}
\item 
\label{genmain1}
either exists a test configuration $(X',L') \to \mathbb C$ for the same polarized manifold with smooth $X'$ and whose central fibre is a reduced simple normal crossing divisor such that $d F(X_0,L_0)<F(X'_0,L'_0)$, for some $d\geq 1$, 
\item 
\label{genmain2}
or else for all $\varepsilon>0$ there is a test configuration $(X',L') \to \mathbb C$ for the same polarized manifold with smooth $X'$ and whose central fibre is a reduced simple normal crossing divisor such that $| d F(X_0,L_0) - F(X'_0,L'_0) | < \varepsilon$, for some $d\geq 1$. 
\end{enumerate}
\end{thm}

An immediate corollary of the above results is then the following:

\begin{thm}\label{thm::smoothing}
Let $(X,L) \to \mathbb C$ be a test configuration for a smooth polarized manifold with $F(X_0,L_0)>0$, then there is a test configuration $(X',L') \to \mathbb C$ for the same polarized manifold with smooth $X'$, whose central fibre is a reduced simple normal crossing divisor, and  $F(X'_0,L'_0)>0$.  
\end{thm}

Hence to check $K$-semistability of a given polarized manifold it is sufficient to restrict
to test configurations with reduced simple normal crossing central fibers.

It is of course of great interest to understand whether the above theorem can be extended to cover the
$K$-polystable case. In this direction we highlight the following partial results:

\begin{thm}\label{thm::zerofut}
Given a test configuration $(X,L) \to \mathbb C$ for a smooth polarized manifold $(M,A)$ with $F(X_0,L_0)=0$ , then:
\begin{enumerate}
\item 
\label{basech}
there exists a test configuration $(X',L') \to \mathbb C$ for the same polarized manifold with 
$F(X'_0,L'_0)=0$ and reduced central fibre;
\item
\label{zerofutnn}
if $X_{\rm non-normal}$ (i.e. the set of non-normal points of $X$) has codimension one, then there exists a smooth test configuration $(X',L') \to \mathbb C$ for the same polarized manifold with $F(X'_0,L'_0)>0$ and reduced simple normal crossing central fibre. In particular if $(M, A)$ is $K$-semistable, then $X_{\rm non-normal}$ has codimension greater than $1$.
\end{enumerate}
\end{thm}

Theorems  \ref{genmain},  \ref{thm::smoothing}  and \ref{thm::zerofut} are proved in Section $3$.

When looking for the worst test configuration, i.e. the one with the highest Futaki invariant, one is forced to introduce some normalization in order to avoid the possibility of arbitrarily enlarging the Futaki invariant (e.g., by coverings of the base). Futaki-Mabuchi \cite{fm}, Sz\'ekelyhidi \cite{Sze06} and Donaldson \cite{do2} have proposed a natural normalization which will be recalled in Section \ref{sec::applications}. Our results then imply that an optimal test configuration in this sense has only reduced simple normal crossing central fibers, or its normalized Futaki invariant can be arbitrarily approximated by the one of test configurations with only reduced simple normal crossing central fiber.
 
\begin{thm}\label{lb}
Let $\Psi(X,L)$ be Donaldson's normalized Futaki invariant (see \cite{do2} and section \ref{sec::applications}). Then
\vspace{-3mm}
 \begin{multline*} \sup \{\Psi(X,L) \mid (X,L) \mbox{ is a test configuration of (M,A)}\} = \\ =  \sup \{\Psi(X',L') \mid (X',L') \mbox{ is a \emph{smooth} test configuration} \\ \mbox{of $(M,A)$ with reduced simple normal crossing central fiber}\}. \end{multline*} 
\end{thm}

Donaldson \cite{do2} proved that the number $\sup \{\Psi(X,L) \mid (X,L) \mbox{ is a test configuration of $(M,A)$}\}$ gives a lower bound for the Calabi energy in the first Chern class of $A$, and conjectures that this lower bound is exact (this has been verified by Sz\'ekelyhidi \cite{sz} for toric varieties assuming the long time existence of the Calabi flow). Thus Theorem \ref{lb} implies that even in seeking to achieve the lower bound of the Calabi energy we can restrict ourselves to test configurations with mild singularities.

In order to understand the relationship between K-stability and Minimal Model Program stability, among all birational transformations we need to understand the behavior of the Futaki invariant under very specific birational transformations: flips, flops and divisorial contractions. The first two are studied in  Proposition \ref{prop::Futakiflips}.

\vspace{.5mm}
{\small \bf Acknowledgements: }The authors would like to thank the referees whose suggestions considerately improved the presentation of the paper.

\section{Extension of Donaldson-Futaki invariant and CM-line to big and nef line bundles}

\noindent As mentioned in the introduction, the understanding of the role of singularities for
K-semistability needs a ``good" extension  of the Donaldson-Futaki invariant \cite{Donaldson2002} to the boundary of the ample cone and in particular to nef and big line bundles in such a way to have a continuity when approaching the boundary. The aim of this section is to show that the following definition achieves this goal:

\begin{defn}\label{def::DF}
Let $V$ be a projective variety or scheme endowed with a $\mathbb C^*$-action and let $L$ be a big and nef line bundle on $V$. Choosing a linearization of the action on $L$ gives a $\mathbb C^*$-representation on $\bigoplus_{j=0}^{\dim V} H^j(V,L^k)^{(-1)^j}$, where we indicate by $W^{-1}$ the dual of $W$ for a vector space $W$. We set $w(V,L^k)={\rm tr}\,A_k$, where $A_k$ is the generator of that representation. As $k \to +\infty$ we have the following classical expansion which follows easily from equivariant Riemann-Roch
$$\frac{w(V,L^k)}{\chi(V,L^k)} = F_0k + F_1 + O(k^{-1}), $$ and we define $$ F(V,L) = F_1$$ to be the \emph{Donaldson-Futaki invariant} of the chosen action on $(V,L)$.
\end{defn}

\begin{rem}
Clearly $w(V,L^k)$ is nothing but the weight of the induced $\mathbb C^*$-action on the determinant
$ \bigotimes_{j=0}^{\dim V} \det H^j(V,L^k)^{(-1)^j}$.
\end{rem}

\begin{rem}
As it is well known, when $L$ is ample $H^q(V,L^k)=0$ for $q\geq 1$, thus we recover Donaldson's definition of the Futaki invariant. In particular $F(V,L)$ is calculated by means of the induced actions on the spaces of sections $H^0(V,L^k)$ for $k \gg 0$. On the other hand, this is not possible in general if $L$ is merely nef and big. To see that, even if we assume the existence of the asymptotic expansion $h^0(V,L^k)= a_0 k^n + a_1k^{n-1} + O(k^{n-2})$ as $k \to +\infty$ (here $n=\dim V$)-- which holds for instance in the case of $L$ being also semi-ample-- we have $a_0>0$ by bigness and $h^q(V,L^k)= O(k^{n-q})$ for $q>0$ by nefness \cite[Theorem 1.4.40]{Laz2004}, so that $h^1(V,L^k)=a_1'k^{n-1} + O(k^{n-2})$ with $a_1'\geq 0$. Analogously the weights of the induced actions on $H^0(V,L^k)$ and $H^1(V,L^k)$ are given respectively by $b_0k^{n+1} + b_1k^n + O(k^{n-1})$ and $b_1'k^n+O(k^{n-1})$, thus by definition \ref{def::DF} we have $$ F(V,L^k) = \frac{a_0(b_1-b_1') - (a_1-a_1')b_0}{a_0^2} = \frac{a_0b_1 - a_1b_0}{a_0^2} - \frac{a_0b_1' - a_1'b_0}{a_0^2}.$$ Hence for such line bundles $L$ on $V$, the Donaldson-Futaki invariant can be computed by means of the induced actions on $H^0(V,L^k)$ under the additional hypothesis $h^1(V,L^k) = O(k^{n-2})$. This represents the crucial technical difference in the definition of the Futaki invariant in this paper and the one given by Ross and Thomas which takes into account just the contribution given by $H^0(V,L)$. 
\end{rem}

To show the mentioned desired properties of this invariant we recall the definition of the (refined) CM-line bundle of a family given by Paul-Tian \cite{pt}.

Let $f: X \to B$ be a family of $n$-dimensional projective schemes. When using the term ``family'' referred to $f: X \to B$, we will always mean that $f$ is a flat projective morphism and more precisely we are given an embedding $i: X \hookrightarrow \mathbb P^N \times B$ such that $f=pr_B \circ i$. Let $L=i^* \circ pr_{\mathbb P^N}^* \mathcal O_{\mathbb P^N}(1)$ be the restriction to $X$ of the obvious relatively (very) ample line bundle on $\mathbb P^N \times B$. Thanks to the relative ampleness of $L$, the comology of the fiber $H^0(X_b,L_b^k)$ is isomorphic to the fiber over $b\in B$ of the direc image $f_*(L^k)$, at least as $k \gg 0$. With this assumption, by flatness the direct image $f_*(L^k)$ is a locally free sheaf on $B$ \cite[Proposition 7.9.13]{EGAIII2}, moreover by relative ampleness we have $H^q(X_b,L_b^k)=0$ for all $q>0$ so that the fiber $f_*(L^k)_b$ is naturally isomorphic to $H^0(X_b,L_b^k)$ \cite[Theorem 12.11]{H77}. Thus we conclude that 
\begin{eqnarray}\label{eq::rank=P} \chi(X_b,L_b^k) &=& \rk f_*(L^k) \\ 
\bigotimes_{j=0}^{\dim X_b} \det H^j(X_b,L_b^k)^{(-1)^j} &=& \det f_*(L^k)|_b, \end{eqnarray} 
for all $b\in B$ and $k \gg 0$. Since $\chi(X_b,L_b^k)$ is a polynomial we have an expansion 
\begin{equation}\label{eq::rk_exp} \rk f_*(L^k) = a_0 k^n + a_1 k^{n-1} + \dots + a_n, \quad \mbox{ as } k\gg 0.\end{equation} 
Now consider the determinant of the locally free sheaf $f_*(L^k)$ for $k$ big enough. As an easy corollary of a result due to Knudsen and Mumford \cite[Proposition 4]{KM76}, we have
\begin{equation}\label{eq::KM_exp} \det f_*(L^k) = \mu_0^{k^{n+1}} \otimes \mu_1^{k^n} \otimes \dots \otimes \mu_{n+1}, \end{equation}
where $\mu_0, \dots, \mu_{n+1}$ are $\mathbb Q$-line bundles on $B$. Combining \eqref{eq::KM_exp} and \eqref{eq::rk_exp}, always for $k\gg 0 $, we get the asymptotic expansion $$ \det f_*(L^k) ^\frac{1}{\rk (X,L^k)} = \mu_0^\frac{k}{a_0} \otimes \left(\mu_1^{a_0} \otimes \mu_0^{-a_1} \right)^\frac{1}{a_0^2} \otimes O(\frac{1}{k}).$$

Up to the factor $-2 a_0 (n+1)!$ the CM-line associated to the family $(X,L) \to B$ defined by Paul and Tian \cite{pt} is the $\mathbb Q$-line bundle on $B$ given by the degree zero term of the expansion above $$ \lambda_{\CM} (X,L) = \left(\mu_1^{a_0} \otimes \mu_0^{-a_1} \right)^\frac{1}{a_0^2}.$$

It is an easy matter to verify that $\lambda_{\CM}(X,L^m) \simeq \lambda_\CM(X,L)$, thus $\lambda(X,L)$ is defined also when $L$ si merely relatively ample.

\bigskip 

Next we want to consider line bundles on $X$ which are not necessarily relatively ample. 

As above, let $f:X \to B$ be a family of $n$-dimensional projective schemes with $n \geq 1$ and let $L$ be a line bundle on $X$. Since $f$ is projective and $L$ can be considered as a perfect complex of sheaves on $X$ supported in degree zero, then the Euler characteristic of $L^k$ restricted to a fiber of $f$ is independent of the chosen fiber and is equal to the rank $\rk Rf_*(L^k)$ of the derived push-forward of $L^k$, thus we have the polynomial expansion 
\begin{equation}\label{eq::rank_exp} \rk Rf_*(L^k) = a_0k^n + a_1k^{n-1} + \dots + a_n, \end{equation} 
with $a_i \in \mathbb Q$. In the following we will be interested mainly in line bundles for which we the term $a_0$ in the polynomial expansion above in non-zero. It is a standard fact that for instance this hypothesis is verified when $L$ is relatively ample or merely relatively big and nef.

Analogously, the determinant $\det Rf_* (L^k)$ of the derived push-forward of $L^k$ has a \emph{polynomial} expansion in terms of some fixed line bundles on the base $B$. More precisely the following holds \cite[Proposition 4]{KM76}:

\begin{thm}[Knudsen-Mumford]\label{thm::Knu-Mum} 
There are line bundles $\nu_i$ on $B$, depending on $f$ and $L$, such that $$ \det Rf_* (L^k) = \nu_0^{\binom{k}{n+1}} \otimes \nu_1^{\binom{k}{n}} \otimes \dots \otimes \nu_{n+1}. $$ 
\end{thm}
This clearly implies the existence of $\mathbb Q$-line bundles $\mu_i$ on $B$ such that \begin{equation}\label{eq::det_exp} \det Rf_* (L^k) = \mu_0^{k^{n+1}} \otimes \mu_1^{k^n} \otimes \dots \otimes \mu_{n+1}. \end{equation}

In order to define the CM-line bundle of the given family, consider the following expansion coming from \eqref{eq::det_exp} and \eqref{eq::rank_exp} as $k \to +\infty$ 
\begin{equation}\label{eq::det^(1/rk)} \det Rf_* (L^k) ^{\frac{1}{\rk Rf_* (L^k)}} = \mu_0^{\frac{k}{a_0}} \otimes \left( \mu_1^{a_0} \otimes \mu_0^{-a_1}\right)^\frac{1}{a_0^2} \otimes O(\frac{1}{k}). \end{equation}

\begin{defn}\label{defn::DFext}
In the situation above, the CM-line associated to the family $(X,L)$ is the $\mathbb Q$-line bundle on $B$ given by $$ \lambda_{\CM} (X,L) = \left( \mu_1^{a_0} \otimes \mu_0^{-a_1}\right)^\frac{1}{a_0^2}.$$
\end{defn}

\begin{rem}
Cleary, $\lambda_\CM(X,L)$ depends on the morphism $f$ and the base $B$ as well. If it is not clear from the context, we shall denote the CM-line bundle by $\lambda_\CM(X/B,L)$.
\end{rem}

We collect in the next proposition the main properties of the CM-line bundle.

\begin{prop}\label{prop::CM_properties}
In the situation above we have
\begin{enumerate}
\item \label{item::hom} $\lambda_\CM(X,L^r)=\lambda_\CM(X,L)$ for all $r>0$,
\item \label{item::nat} if $\Lambda$ is a line bundle on $B$, then $\lambda_\CM(X,L\otimes f^*\Lambda) = \lambda_\CM(X,L)$,
\item \label{item::fiber} if $f': X' \to B$ is another flat family endowed with a relatively ample line bundle $L'$, then $$ \lambda_\CM (X \times_B X', L\boxtimes L') = \lambda_\CM(X,L) \otimes \lambda_\CM(X',L'),$$
\item \label{item::base} if $\phi:B' \to B$ is flat and 
\begin{equation*}
\xymatrix{ 
X \times_B B' \ar[d]^g \ar[r]^p & X \ar[d]^f \\
B' \ar[r]^\phi & B }
\end{equation*}   
is the base change induced by $\phi$, then $$\lambda_\CM(X \times_B B'/B', p^*L) = \phi^*\lambda_\CM(X/B,L).$$
\item \label{item::sbir}If $(X',L')\to B$ is another family and $\xi:X \to X'$ is a small {\begin{footnote} {Recall that a birational morphism $f:X\to Y$ is said to be small if, when denoting its exceptional locus by $Ex(f)$, one has that $codim Ex(f) \leq 2$.} \end{footnote}}  birational (regular) morphism such that $\xi^*(L')=L$ and the diagram 
\begin{equation*}
\xymatrix{ 
X \ar[d]^f \ar[r]^\xi & X' \ar[dl]^{f'} \\
B & }
\end{equation*} is commutative, then $\lambda_\CM(X,L) = \lambda_\CM(X',L').$
\end{enumerate} 
\end{prop}

\begin{proof}
Assertion \ref{item::hom} is obvious from \eqref{eq::det^(1/rk)}. Assertions \ref{item::nat} and \ref{item::fiber} are proved in \cite{fr}, but \ref{item::nat} follows readily from \eqref{eq::det^(1/rk)} and the fact that:
$$ \det Rf_* (L^k \otimes f^*\Lambda^k) = \det Rf_*(L^k) \otimes \Lambda^{k\,\rk Rf_*(L^k)},$$
where we used projection formula and the identity $\det (F \otimes A) = \det F \otimes A^{\rk F}$, for any vector bundle $F$ and line bundles $A$ on $B$. 
In order to prove (\ref{item::base}) we notice that \cite[Proposition 9.3]{H77} implies $ Rg_*(p^*L^k) = \phi^* Rf_*(L^k) $, whence:
 $$\det Rg_* (p^*L^k)^\frac{1}{\rk Rg_* (p^*L^k)} = \phi^* \det Rf_*(L^k)^\frac{1}{\rk Rf_*(L^k)},$$ and the thesis follows.

Finally, in the situation of (\ref{item::sbir}), having recalled that $\xi$ is called small if its exceptional locus is of codimension at least two, by projection formula we get:
$$ Rf_* (L^k) = R(f'\circ \xi)_* \left( \xi^*(L')^k \right) = R f'_* \left( R\xi_*(\mathcal O_X) \otimes (L')^k\right).$$ 
Now consider the exact sequence:
 $$ 0 \to \mathcal O_{X'} \to R\xi_*(\mathcal O_X) \to Q \to 0, $$ where $\codim \supp Q \geq 2$ thanks to the smallness of $\xi$. Thus, after tensoring by $(L')^k$ and taking the derived direct image via $f'$, thanks to the additivity properties of $\det$ and $\rk$ we get:
  $$ \det Rf'_*\left( (L')^k\right) ^ \frac{1}{\rk Rf'_*\left( (L')^k\right)} = \det Rf_*\left( L^k \right) ^ \frac{1}{\rk Rf_*\left( L^k\right)} \otimes O\left( \frac{1}{k}\right)$$ and the statement readily follows from the definition of $\CM$-line bundle. 
\end{proof}


Moreover, the CM-line bundle has a sort of continuity property w.r.t. the line bundle $L$. More precisely the following holds:

\begin{prop}\label{prop::CM_continuity}
Let $L$ and $N$ be two line bundles on $X$ and suppose $L$ relatively big and nef as above and $N$ relatively ample w.r.t $f: X \to B$. We have $$ \lambda_\CM(X,L^r\otimes N) = \lambda_\CM(X,L) \otimes O(\frac{1}{r}) \qquad \mbox{ as } r \to \infty.$$
\end{prop}

\begin{proof} By assertions \ref{item::hom} and \ref{item::nat} of proposition \ref{prop::CM_properties} we have $$ \lambda_\CM(X,L^r\otimes N) = \lambda_\CM(X,L^{sr} \otimes N^s \otimes f^*\Lambda) $$ for all $s>0$ and any line bundle $\Lambda$ on $B$. In particular, taking $\Lambda$ sufficiently ample, thanks to \cite[Proposition 1.7.10]{Laz2004} we may assume without loss of generality $N$ to be very ample on $X$. For each $k \gg 0$, let $\sigma_1, \dots, \sigma_k \in H^0(X,N)$ be general sections. Denoting by $Z_i$ the null scheme of $\sigma_i$, we have the following exact sequence (where the first map is given by multiplication by $\sigma_1 \otimes \dots \otimes \sigma_k$):
 \begin{multline*} 0 \to L^{rk} \to (L^r\otimes N)^k \to \bigoplus_{i=1}^k (L^r\otimes N)^k \otimes \mathcal O_{Z_i} \to \bigoplus_{1 \leq i_0 < i_1 \leq k} (L^r\otimes N)^k \otimes \mathcal O_{Z_{i_0} \cap Z_{i_1}} \to \dots \\ \dots \to \bigoplus_{1 \leq i_0 < \dots < i_n \leq k} (L^r\otimes N)^k \otimes \mathcal O_{Z_{i_0} \cap \dots \cap Z_{i_n}} \to 0,\end{multline*}
 whence:
  \begin{multline*} \rk Rf_*\left( L^{rk}\otimes N^k \right) = \rk Rf_* (L^{rk}) + \\ + \sum_{\ell=0}^{n} (-1)^\ell \rk Rf_*\left( \bigoplus_{1 \leq i_0 < \dots < i_\ell \leq k} (L^r\otimes N)^k \otimes \mathcal O_{Z_{i_0} \cap \dots \cap Z_{i_\ell}}\right) \end{multline*}
and:
 \begin{multline*} \det Rf_*\left( L^{rk}\otimes N^k \right) = \det Rf_* (L^{rk}) \otimes \\ \otimes \bigotimes_{\ell=0}^{n} \left( \det Rf_*\left( \bigoplus_{1 \leq i_0 < \dots < i_\ell \leq k} (L^r\otimes N)^k \otimes \mathcal O_{Z_{i_0} \cap \dots \cap Z_{i_\ell}}\right)\right)^{(-1)^\ell}. \end{multline*}
Since $$ \rk Rf_*\left(\bigoplus_{i_0=0}^k  (L^r\otimes N)^k\otimes \mathcal O_{Z_{i_0}}\right) = c_0r^{n-1} k^n + O(r^{n-2}),$$ $$ \rk Rf_*\left(\bigoplus_{1 \leq i_0 \leq \dots \leq i_\ell \leq k} (L^r\otimes N)^k\otimes \mathcal O_{Z_{i_0}\cap \dots \cap Z_{i_\ell}}\right) = O(r^{n-2}) \qquad \mbox{for all }\ell \geq 1,$$
and analogously: 
$$ \det Rf_*\left(\bigoplus_{i_0=0}^k  (L^r\otimes N)^k\otimes \mathcal O_{Z_{i_0}}\right) = \rho_0^{r^n k^{n+1}} \otimes O(r^{n-1}),$$ $$ \det Rf_*\left(\bigoplus_{1 \leq i_0 \leq \dots \leq i_\ell \leq k} (L^r\otimes N)^k\otimes \mathcal O_{Z_{i_0}\cap \dots \cap Z_{i_\ell}}\right) = O(r^{n-1}) \qquad \mbox{for all }\ell \geq 1,$$
we have (here we get expansions for $\rk Rf_*(L^{rk})$ and $\det Rf_*(L^{rk})$ from \eqref{eq::rank_exp} and \eqref{eq::det_exp}) 
$$ \rk Rf_*(L^{rk}\otimes N^k) = (a_0r^n + c_0 r^{n-1})k^n + a_1 r^{n-1}k^{n-1} + O(r^{n-2}), $$ 
$$ \det Rf_*(L^{rk}\otimes N^k) = \left(\mu_0^{r^{n+1}} \otimes \rho_0^{r^n}\right)^{k^{n+1}} \otimes \mu_1^{r^nk^n} \otimes O(r^{n-1}). $$
Thus:
\begin{eqnarray*} \lambda_\CM(X,L^r\otimes N) &=& \left( \mu_1^{a_0r^{2n}} \otimes 
\mu_0^{-a_1r^{2n}} \otimes O(r^{2n-1}) \right)^\frac{1}{(a_0r^n +O(r^{n-1}))^2} \\ &=& \left( \mu_1^{a_0} \otimes \mu_0^{-a_1} \right)^\frac{1}{a_0^2} \otimes O\left(\frac{1}{r}\right), \end{eqnarray*} and we are done. 
\end{proof}


\section{Applications}\label{sec::applications}

In this section we suppose that the polarized family $f: (X,L) \to B$ of the previous section is a test configuration for a smooth manifold as defined by Donaldson \cite{Donaldson2002}. This means that $B=\mathbb C$ and we are given a $\mathbb C^*$-action on $X$ that linearizes to $L$ and covers the standard action on $\mathbb C$, making $f$ an equivariant map. Moreover the fiber $X_t=f^{-1}(t)$ is smooth for all $t \neq 0$ (see Definition \ref{defn::TC}). In this situation the expansion \eqref{eq::det_exp} holds in the sense of linearized $\mathbb Q$-line bundles, thus the CM-line bundle $\lambda_\CM(X,L)$ comes equipped with a linearization. Moreover, Proposition \ref{prop::CM_properties} holds, \emph{mutatis mutandis}, in the sense of linearized line bundles; in particular, property \ref{item::nat} implies that the linearization on $\lambda_\CM(X,L)$ is independent of the one chosen on $L$.
The central fiber $(X_0,L_0) = (f^{-1}(0), L|_{f^{-1}(0)})$ is equipped with a $\mathbb C^*$-action, since it lies over the fixed point $0 \in \mathbb C$. In case of $L$ ample, the relation between the CM-line bundle and the Donaldson-Futaki invariant $F(X_0,L_0)$ is given by the following \cite{pt}

\begin{prop}[Paul-Tian]
The weight of the $\mathbb C^*$-action induced on the fiber of $\lambda_\CM(X,L)$ over $0\in \mathbb C$ equals the Donaldson-Futaki invariant $F(X_0,L_0)$ of the central fiber.
\end{prop}

In order to extend this result to not necessarily relatively ample line bundles we need to use the following theorem essentially due to Knudsen-Mumford \cite{KM76}:

\begin{thm}\label{thm::res_rk_det}
Let $f: X \to Y$ be a projective morphism of schemes, and let $\mathcal F$ be a perfect complex on $X$. We have
$\bigotimes_{j=0}^{\dim X_y} \det H^j(X_y,\mathcal F_y)^{(-1)^j} \simeq \det R f_*(\mathcal F)|_{y}$ functorially for all $y \in Y$.  
\end{thm}

%
We can then show:
\begin{prop}\label{prop::DF_continuity}
Let $L$ be a relatively big and nef line bundle on $X$. The weight of the $\mathbb C^*$-action induced on the fiber of $\lambda_\CM(X,L)$ over $0\in \mathbb C$ equals the Donaldson-Futaki invariant $F(X_0,L_0)$ (as defined in \ref{def::DF}) of the central fiber.
\end{prop}

\begin{proof}
Since $L$ is $\mathbb C^*$-linearized, the determinant $\det R f_* (L^k)$ inherits a $\mathbb C^*$-linearization. Regarding $L$ as a perfect complex (supported on degree zero) on $X$, by Theorem \ref{thm::res_rk_det} we obtain an equivariant isomorphism 
$$ \bigotimes_{j=0}^{\dim X_0} \det H^j(X_0,L_0^k)^{(-1)^j} \simeq \det Rf_* (L^k)|_0 $$ for each $k>0$.

By \eqref{eq::det_exp} we have an equivariant expansion 
$$  \bigotimes_{j=0}^{\dim X_0} \det H^j(X_0,L_0^k)^{(-1)^j} \simeq \mu_0|_0^{k^{n+1}} \otimes \mu_1|_0^{k^n} \otimes \dots \otimes \mu_{n+1}|_0,$$ whose weight must coincide for every $k$ with $$ w(X_0,L_0^k) = b_0 k^{n+1} + b_1 k^n + \dots + b_{n+1}.$$ 
Hence the weight on the $\mathbb Q$-line $\mu_j|_0$ is $b_j$ and the thesis follows by definition \ref{defn::DFext}. 
\end{proof}

\begin{cor}
\label{continue}
Let $L,A$ be linearized line bunldes on a scheme $V$ acted on by $\mathbb C^*$. Suppose that $L$ is big and nef and $A$ ample. We have $$ F(V,L^r \otimes A) = F(V,L) + O\left( \frac {1}{r} \right), \qquad \mbox{as }r \to \infty.$$ 
\end{cor}

We now need to recall the following (cf. \cite{ko-mo} Def. 3.33 page 99 and Def. 6.10 page 191):

\begin{defn}
Let $f:X\to Y$ be a proper small birational morphism, and assume that $D$ is a divisor such that $K_X+D$ is $\QQ$-Cartier and $-K_X$ is $f$-ample (resp. numerically $f$-trivial). A variety $X^+$ along with a {\it small} proper morphism $f^+:X^+\to Y$ (which then induces a birational map $\phi: X\to X^+$) is called a $K_X+D$-{\it flip} (resp. $D$-{\it flop}) -or simply flip, when $D=\emptyset $- if:
\begin{enumerate}
\item $K_{X^+}+D^+$ is  $\QQ$-Cartier, if $D^+$ denotes the closure of $\phi^{-1}(D)$
\item $K_{X^+}+D^+$ is $f^+$-ample.
\end{enumerate} 
\end{defn}

Combining Propositions \ref{prop::DF_continuity} and \ref{prop::CM_properties} yields: 

\begin{prop}\label{prop::Futakiflips}
Given two test configurations $(X,L)$ and $(X',L')$ and $\xi: X\to X'$ a $\CC^*$-equivariant small birational morphism such that $\xi ^* (L')=L$, then $F(X_0, L_0)= F(X_0', L_0')$. In particular, the Futaki invariant is preserved under  $K_X+D$-{\it flips} and  $D$-{\it flops} of the family which preserve the generic fibre.
\end{prop}
\begin{proof}
The first part of the proposition follows directly from Propositions \ref{prop::DF_continuity} and \ref{prop::CM_properties}. 
Therefore, all one needs to show is that a flip or flop of a test configuration stays such, and hence simply that a flip or flop $X^+$ of $X$ is still endowed with a $\CC^*$-action. This is easy to show, as in fact (after an argument involving a $\CC^*$-equivariant Hironaka) one can reduce oneself to considering a projective scheme $W$ endowed with a $\CC^*$-action and a regular birational morphism $\phi^+:W\to X^+$. 
It is now easy to show, given any $f^+$-ample line bundle $A$ on $X^+$, that the $\CC^*$-action on $W$ induces an action on $H^0(X^+, A^k)$ for any integer $k$ which coincides with the natural action on $H^0(W,(\phi^+)^*A ^k)$. Hence, taking $A= K_{X^+} +D^+$, we find there is a $\CC^*$-action on $X^+= \proj _Y\left( \bigoplus_k H^0(X^+, A^k)\right)$ which coincides with the action on $W$ on the Zariski open sets on which $f^+$ is an isomorphism.
\end{proof}
Before stating our main result we need to recall two important results. The first one is essentially due to Mumford \cite{KKMS73}

\begin{thm}[Equivariant semi-stable reduction]\label{thm::ESSR}
Let $f:X \to \mathbb C$ be a $\mathbb C^*$-equivariant family of projective schemes with smooth general fiber. Then there exist an integer $d>0$ and a projective equivariant morphism $\beta$ as follows 
$$
\xymatrix{ X' \ar@{.>}[ddr]_{f'} \ar[dr]^{\beta} & & \\
& X \times_{\pi_d} \mathbb C \ar[d] \ar[r] & X \ar[d]^f \\
& \mathbb C \ar[r]^{\pi_d} & \mathbb C }
$$
where $\pi_d(z) = z^d$, such that 
\begin{itemize}
\item $\beta$ is the blow-up of an invariant ideal sheaf supported over $0 \in \mathbb C$, 
\item the square is equivariant if we compose the given $\mathbb C^*$-action on $f: X \to \mathbb C$ with the $d$-fold covering $t \mapsto t^d$ on $\mathbb C^*$. 
\item $X'$ is smooth and the central fibre $f'^{-1}(0)$ is a reduced with non-singular components crossing normally.  
\end{itemize}
\end{thm}

\begin{proof}
For the time being, let us neglect the $\mathbb C^*$-action. Applying Mumford's semi-stable reduction theorem \cite{KKMS73}, we get a smooth curve $C'$ with a marked point $0'$, a finite morphism $\pi: C' \to \mathbb C$ such that $\pi^{-1}(0) = \{0'\}$, and a projective morphism $\beta$ as follows     
$$
\xymatrix{ X' \ar@{.>}[ddr]_{f'} \ar[dr]^{\beta} \\
& X \times_{\pi} \mathbb C \ar[d] \ar[r] & X \ar[d]^f \\
& C' \ar[r]^{\pi} & \mathbb C }
$$
such that $\beta$ is an isomorphism over $C'\setminus \{0'\}$, $X'$ is smooth and the fiber $f'^{-1} (0') $ is reduced with non-singular components crossing normally. 

Now we show that everything can be supposed $\mathbb C^*$-equivariant. First of all, restricting $\pi$ to a local chart $\mathbb C$ centered at  $0'\in C'$, we may suppose without loss of generality that $\pi = \pi_d$ for some integer $d>0$. Then if we compose the given action on $X$ and $\mathbb C$ with the $d$-th covering $t \mapsto t^d$ of $\mathbb C^*$, we obtain a new action on $f:X \to \mathbb C$ inducing an action on the fiber product $X \times_{\pi_d}\mathbb C$ that makes the projections over $X$ and $\mathbb C$ equivariant. Finally, since the existence of $\beta$ is a consequence of  Hironaka's resolution theorem, we can suppose $X'$ acted on by $\mathbb C^*$ and $\beta$ equivariant thanks to the equivariant resolution theorem (cf. \cite{ko}, 4.1 pg.4).       
\end{proof}

The second result we need is the following proposition which has been proved in  \cite[Proposition 5.1]{rt} by Ross-Thomas with the classical definition of Futaki invariant.

\begin{prop}\label{prop::dominate_TC}
Given a test configuration $f:(X,L) \to \mathbb C$ as above, let $f':(X',L')\to \mathbb C$ be another flat equivariant family with $X'$ normal and let $\beta: (X',L') \to (X,L)$ be a $\mathbb C^*$-equivariant birational map such that $f'=f \circ \beta$ and $L'=\beta^*L$. Then we have $$ F(X'_0,L'_0) \geq F(X_0,L_0), $$ with strict inequality if and only if the support of $\beta_*(\mathcal O_{X'})/\mathcal O_X$ has codimension one.  
\end{prop}

\begin{proof}
Since our definition of Futaki invariant involves higher cohomology, the statement is not a priori the same as the one by Ross-Thomas (loc. cit.). On the other hand we prove the statement reducing to the situation considered by Ross-Thomas.   
For each $m \in \mathbb Z$ using the projection formula we have:
\begin{eqnarray*} Rf'_*((L')^m) &=& R(f_*\circ \beta_*) (\beta^*L^m) = Rf_* \circ R \beta_* (\beta^* L^m) \\ &=& Rf_*(R\beta_*(\mathcal O_{X'}) \otimes L^m) ,\end{eqnarray*}

Now, by \cite[Proposition 8]{KM76} $$ \det Rf'_*((L')^m) = 
\bigotimes_{h,k} \det (R^hf_*(H^k(R\beta_* (\mathcal O_{X'}) \otimes L^m)))^{(-1)^{h+k}}. $$ 
Since  $L$ is relatively ample, we can consider $L^m$ as a complex supported in degree $0$, so we have:
$$H^k(R\beta_* (\mathcal O_{X'}) \otimes L^m)= R^k\beta_* (\mathcal O_{X'}) \otimes L^m$$
and therefore using the relative ampleness of $L$ again (through Serre's criterion of ampleness):
$$ R^hf_*(H^k(R\beta_* (\mathcal O_{X'}) \otimes L^m))=  R^hf_*(R^k\beta_* (\mathcal O_{X'}) \otimes L^m)=0$$
if $h>0$, when $m\gg1$.
hence:
 $$ \det Rf'_*((L')^m) = \bigotimes _k \det f_*(R^k\beta_*(\mathcal O_{X'})\otimes L^m)^{(-1)^{k}}.$$
We now claim that since $X'$ is normal, the term $ \det f_*(R^k\beta_*(\mathcal O_{X'})\otimes L^m)$ is $ O(m^{n-1})$ where $n+1$ is the dimension of $X'$ and $k>0$. Indeed, we can apply  \cite[Theorem 4]{KM76}, to $Y=\CC$,  $P=\mathbb P (E)$ with $E=\CC \times  H^0( X, L^m)^*$ (recall $X\to \CC$ is relatively imbedded in $\PP ^{N_m}=\mathbb P H^0( X, L^m)^*$ via $L^m$) and with $\mathcal F = R^k\beta_*(\mathcal O_{X'}) \otimes \mathcal O _P$ (hence  $\mathcal F(m)=R^k\beta_*(\mathcal O_{X'}) \otimes \mathcal O _P(m) \simeq (R^k\beta_*(\mathcal O_{X'})\otimes L^m)\otimes \mathcal O _P$).
Now note that $\mathcal F(m)$ restricts to $R^k\beta_*(\mathcal O_{X'})\otimes L^m$ on $X$ and that $\det$ commutes with base change.

One has that $ \det  f_*(R^k\beta_*(\mathcal O_{X'})\otimes L^m)$ grows like $m^{r+1}$, where $r+1$ is the dimension of the scheme-theoretic intersection $ \supp \left( R^k\beta_*(\mathcal O_{X'})\right) \cap f^{-1}(0)$. Indeed in the present situation, from Definition of property $Q_{(r)}$ in \cite[pg. 50]{KM76}, the number $r$ is defined by $$r = \min\left\{ s>0 \,|\,\dim \supp \left( R^k\beta_*(\mathcal O_{X'})\right)_{y}\leq s+{\rm depth}(y) \mbox{ for all } y \in \CC\right\},$$ where $ \supp \left( R^k\beta_*(\mathcal O_{X'})\right)_y= \supp R^k\beta_*(\mathcal O_{X'}) \times _{\CC} {\rm Spec}(\CC(y))$ is the scheme-theoretic intersection $ \supp \left( R^k\beta_*(\mathcal O_{X'})\right) \cap f^{-1}(y)$ for every point (geometric or generic) $y\in \CC$. All geometric points of $\CC$ have depth 1 and the null ideal of $\CC$, which is the only generic point, has depth 0. On the other hand $R^k\beta_*(\mathcal O_{X'})$ is supported over $0 \in \CC$ thus $\dim \supp \left( R^k\beta_*(\mathcal O_{X'})\right)_y$ can be positive only over the geometric point $y=0$, thus we get $r+1=\dim \supp \left( R^k\beta_*(\mathcal O_{X'})\right)_0$ and therefore the claim reduces to showing that the codimension of  $\supp \left(R^k\beta_*(\mathcal O_{X'})\right)_0$ in $f^{-1}(0)$ is at least $1$, or that the codimension of  $\supp \left(R^k\beta_*(\mathcal O_{X'})\right)$ in $X$ is at least $2$.

We argue this as follows. Note that since $X'$ is normal we can factor $\beta$ through the normalization $\nu :W\to X$. So we have a diagram:
$$X'   \stackrel{\beta '}{\to}W \stackrel{\nu}{\to} X$$ with $\beta = \nu \circ \beta'$. Since $\nu$ is finite (and so has no higher cohomology) $R\beta_* (\mathcal O_{X'}) = R\nu_* \circ R\beta_* ' (\mathcal O_{X'}) = \nu_* \circ R\beta'_*\left( \mathcal O_{X'}\right)$, thus the support of $R^k\beta_*(\mathcal O_{X'})$ is contained in  the image via $\nu$ of the support of $R^k\beta_* '(\mathcal O_{X'})$, which in turn is contained it $\sing(W)$ hence it has codimension at least $2$ since $W$ is normal. Indeed if $\delta:X''\to X'$ is a desingularization and $\beta''= \beta' \circ \delta $, then we have $R \beta'_* (\mathcal O_{X'}) = R \beta'_*(\delta_* (\mathcal O_{X''}))$, thus $\supp R \beta'_* (\mathcal O_{X'}) \subset \supp R \beta''_*(\mathcal O_{X''}) \subset \sing(W)$, being $\beta''$ a desingularization of $W$. Therefore the leading and the following term of $ \det Rf'_*((L')^m)$ are concentrated in $ \det f_*(\beta_*(\mathcal O_{X'})\otimes L^m)=\det f_* (p_* (\mathcal O _Z) \otimes L^m)$ (in this equality we just use the equivariant isomorphism $\mathcal O_Z \simeq q_*(\mathcal O_{X'})$ coming from the Stein factorization $ X' \stackrel{q}{\to} Z \stackrel{p}{\to} X$ of $\beta$) and we are in the situation considered by Ross-Thomas. In particular the weight 
of the $\mathbb C^*$-action on 
$$ \left( \det Rf'_*((L')^m) \otimes \det f_*(L^m)^{-1} \right)|_0$$ is equal to 
$a\,m^n + O(m^{n-1})$ with $a>0$ when $\supp (\beta_*(\mathcal O_{X'}) / \mathcal O_X)$ has dimension $n$ and $a=0$ otherwise.
\end{proof}

In particular, we can control the behavior of the Futaki invariant under some important class of birational morphisms. 

\begin{cor}\label{cor::control_Futaki}
In the situation of Proposition \ref{prop::dominate_TC} we have 
\begin{enumerate}
\item If $f:X' \to X$ is the blow-up of $X$ along an invariant subscheme of codimension at least two set-theoretically supported over the central fiber $X_0$, we have $F(X_0',L_0') = F(X_0,L_0)$.
\item If $X_{\rm non-normal}$ (i.e. the set of non-normal points of $X$) has codimension at least two, then $F(X_0',L_0') = F(X_0,L_0)$.
\item If $X_{\rm non-normal}$ has codimension one, then $F(X_0',L_0') > F(X_0,L_0)$.     
\end{enumerate}
\end{cor}

\begin{proof}
The first assertion follows easily after noting that $\beta_*(\mathcal O_{X'}) / \mathcal O_X$ can be non-zero only over the center of the blow-up, that has at least codimension two.

To prove assertion two and three observe that thanks to the normality of $X'$, $\beta$ factorizes $\beta = \nu \circ \beta' $ through the (equivariant) normalization $\nu: W \to X$. By Zariski's main theorem $\beta'_* (\mathcal O_{X'}) \simeq \mathcal O_W$, thus $ F(X'_0,L'_0) = F(W_0,\nu^*L|_0)$ and $\beta_*( \mathcal O_{X'} ) \simeq \nu_* (\mathcal O_W)$, whence the thesis follow.   
\end{proof}

We can now prove the main Theorems:
\begin{proof}[Proof of Theorem \ref{genmain}]
Consider the CM-line bundle $\lambda_\CM(X,L)$ and apply the semi-stable reduction theorem \ref{thm::ESSR} to the family $(X,L)$. Since $pr_X:X \times_{\pi_d}\mathbb C \to X$ is a finite map, then $pr_X^*L$ is ample; moreover by assertion \ref{item::base} of Proposition \ref{prop::CM_properties} we get $\lambda_\CM(X \times_{\pi_d}\mathbb C, pr_X^*L) = \pi_d^*\lambda_\CM(X,L)$, thus 
by Proposition \ref{prop::CM_continuity} on the central fibres we have \begin{equation}\label{eq::base_change} F((X \times_{\pi_d}\mathbb C)_0, (pr_X^*L)_0) = d\, F(X_0,L_0),\end{equation} 
where $d$ is the degree of the base change in the reduction.

On the other hand, denoted by $E$ the exceptional divisor of $\beta$, the line bundle  $L'(r)=\beta^*pr_X^*L^r(-E)$ on $X'$ is relatively ample for $r$ big enough and $E$ is trivial outside form central fibre, thus $(X',L'(r)) \to \mathbb C$ is a test configuration for the original polarized manifold. By Corollary \ref{continue} we can approximate the Donaldson-Futaki invariant of the line bundle pulled-back via $\beta$ \begin{equation}\label{eq::blow_up} F(X'_0,L'(r)_0) = F(X'_0, (\beta^*pr_X^*L)_0) + O\left( \frac{1}{r} \right) \quad \mbox{as } r \to \infty,\end{equation} but finally we observe that $$ \beta: (X',\beta^*pr_X^*L) \to (X \times_{\pi_d}\mathbb C, pr_X^*L)$$ satisfies the hypothesis of Proposition \ref{prop::dominate_TC}. Thanks to \eqref{eq::blow_up} the theorem follows. Indeed, if we have $ F(X_0',L'_0) > d F(X_0,L_0)$, then the test configuration $(X', L'(r)) \to \mathbb C$ satisfies \eqref{genmain1} for $r \gg 0$. Otherwise, if $ F(X_0',L'_0) = d F(X_0,L_0)$, then for all $\varepsilon > 0$ there exists $r_0(\varepsilon)>0$ such that $F(X_0',L'(r)_0) > d\, F(X_0,L_0)- \varepsilon$ for all $r>r_0(\varepsilon).$
\end{proof}

\begin{proof}[Proof of Theorem \ref{thm::zerofut}]
To prove \eqref{basech} it is enough to perform a base change of some degree $d$ of the original test configuration, which has the effect of multiplying by $d$ the Futaki invariant.

The second assertion follows from the same argument as in the proof of Theorem \ref{genmain} using the third assertion of Corollary \ref{cor::control_Futaki}.
\end{proof}

As mentioned in the introduction, in order to prevent  the Futaki invariant from arbitrarily increasing in trivial ways, one needs to introduce a certain  
normalization on the space of test configurations. This was achieved first by Futaki-Mabuchi \cite{fm} in the smooth case, then by Sz\'ekelyhidi \cite{Sze06} and
Donaldson \cite{do2} for general schemes, by defining a norm of a test configuration $(X,L)\to \mathbb C$ of an $n$-dimensional smooth polarized manifold $(M,A)$ as follows: $$\| X_0,L_0 \| = \sqrt{\frac{Q}{a_0}-\frac{b_0^2}{a_0^2}} \,\, ,$$

where $Q$ is the leading coefficient of the expansion in $k$
of $\tr A_k^2$ (see definition \ref{def::DF}) referred of course to
the action on the central fiber $(X_0,L_0)$. Then we look at $$ \Psi(X,L) = \sqrt[n]{a_0 \vol(M,A)^{\frac{n-2}{2}}} \frac{F(X_0,L_0)}{\| X_0,L_0\|},$$ where of course $\vol(M,A)= e^{-n}a_0$ if $e$ is the exponent of the test configuration.

\begin{prop}
The function $\Psi$ is invariant by base change $t \mapsto t^d$. In the situation of Proposition \ref{prop::dominate_TC}, we have $$
\Psi(X',L') \geq \Psi(X,L),$$ with strict inequality if and only if the support of $\beta_*(\mathcal O_{X'})/\mathcal O_X$ has codimension one.
\end{prop}

\begin{proof}
The base change $t \to t^d$ (with $d>0$ necessarily) transforms $A_k$ to $d\, A_k$, whence
$\|X_0,L_0\|$ changes to $d\,\|X_0,L_0\|$, and $F(X_0,L_0)$ to $d\, F(X_0,L_0)$.
On the other hand, since $\|X_0,L_0\|$ depends only on leading coefficients of polynomials $\chi(X_0,L_0^k)$, $\tr A_k$, and $\tr A_k^2$, thanks to the proof of Proposition \ref{prop::dominate_TC}, it is unchanged under $\beta$. Thus the statement follows from Proposition \ref{prop::dominate_TC}.
\end{proof}

In light of the Proposition above, the proof of Theorem \ref{lb} is
reduced to a straightforward exercise.

\end{document}